\documentclass[a4paper, 12pt]{article}

\usepackage[top=2.0cm,bottom=2.5cm,left=2.0cm,right=2.0cm]{geometry}
\usepackage{lineno,hyperref}
\usepackage{amsmath,amsthm,amsfonts,amssymb,mathrsfs,graphicx,color}
\usepackage{subfigure,float}
\usepackage{enumerate}
\usepackage[T1]{fontenc}
\usepackage[title]{appendix}

\modulolinenumbers[5]

\begin{document}
\newtheorem{definition}{Definition}[section]
\newtheorem{theorem}{Theorem}[section]
\newtheorem{lemma}[theorem]{Lemma}
\newtheorem{example}[theorem]{Example}
\newtheorem{remark}[theorem]{Remark}
\newtheorem{corollary}[theorem]{Corollary}
\newtheorem{proposition}[theorem]{Proposition}

\title{The net-regular strongly regular signed graphs with degree 5}

\author{Qian Yu\thanks{1779382456@qq.com}, Yaoping Hou\thanks{Corresponding author: yphou@hunnu.edu.cn}\\
{\small MOE-LCSM, CHP-LCOCS, School of Mathematics and Statistics}\\{\small Hunan Normal University, Changsha, Hunan 410081, P. R. China}
}

\date{}
\maketitle

\begin{abstract}
In this paper,  we determine all connected  net-regular strongly regular signed graphs with degree 5. There are five and two strongly regular signed graphs with net-degree 3 and 1, respectively.
\end{abstract}


\section{Introduction}

Let $G$ be a simple connected graph with vertex set $V(G)$ and edge set $E(G)$. A \emph{signed graph} $\dot{G}=(G,\sigma)$ is an unsigned graph $G$ together with a sign function $\sigma :E(G)\rightarrow\{+1,-1\}$. The \emph{adjacency matrix} of $\dot{G}$ is denoted by $A_{\dot{G}}=(a_{ij}^\sigma)$ where $a_{ij}^\sigma=\sigma(v_iv_j)$ if $v_i\sim v_j$, and 0 otherwise. And the eigenvalues of $\dot{G}$ are the eigenvalues of $A_{\dot{G}}$. Let $N(v)$ denote the set of neighbours of $v$ and  $N(u)\bigcap N(v)$ denote the set of common neighbours of $u$ and $v$. The \emph{degree} $d(v)$ of a vertex $v$ in a signed graph $\dot{G}$ is its degree in the corresponding underlying graph.  The \emph{positive degree} $d^{+}(v)$ and \emph{negative degree} $d^{-}(v)$ of $v$ are the number of positive and negative edges incident with $v$, respectively. And we use $v_i\overset{+}{\sim}v_j$ and $v_i\overset{-}{\sim}v_j$ to denote that vertices $v_i, v_j$ are joined by a positive edge and a negative edge, respectively. A signed graph is said to be \emph{$r$-regular}  if its every vertex has the same degree $r$ and \emph{$\rho$ net-regular}  if its every vertex has the same net-degree $\rho$. A signed graph $\dot{G}$ is called \emph{homogeneous} if all its edges have the same sign. Otherwise, it is \emph{inhomogeneous}. We use $G^+$ and  $G^-$ to denote the subgraphs of ${G}$ induced by the positive edges and negative edges of $\dot{G}$, respectively, in this paper. And $-\dot{G}$ denotes the \emph{negation} of $\dot{G}$.

A walk of length $n$ in a signed graph $\dot{G}$ is a sequence of vertices $v_0v_1\cdots v_n$ such that $v_{i-1}$ is adjacent to $v_i$ for each $i=1,2,\dots,n$. And a walk is positive if $\sigma(v_0v_1)\sigma(v_1v_2)\cdots\sigma(v_{k-1}v_k)=1$ and negative if $\sigma(v_0v_1)\sigma(v_1v_2)\cdots\sigma(v_{k-1}v_k)=-1$. A cycle of a signed graph is a closed walk with distinct vertices. And the sign of this cycle is the product of the sign of its edges. A signed graph is \emph{balanced} if all its cycles are positive. Otherwise, it is  \emph{unbalanced}.

We have known that a strongly regular graph with parameters $(n,r,e,f)$ is an $r$-regular graph on $n$ vertices in which any two adjacent vertices have exactly $e$ common neighbours and any two non-adjacent vertices have exactly $f$ common neighbours \cite{CRS}. And the parameters of a strongly regular graph satisfy that
\begin{equation}\label{srg}
   r(r-e-1)=(n-r-1)f.
\end{equation}

 Regarding  the strongly regularity of signed graphs, Zaslavsky introduced the concept of the very strongly regular signed graphs in \cite{Z}. Moreover, Ramezani developed another kind of concept of signed strongly regular graphs and  constructed several families of signed strongly regular graphs with only two distinct eigenvalues by using the star complement technique in \cite{R}. But in this paper, we consider the strongly regular signed graphs defined by Stani\'c in \cite{S1}.

By \cite{S1}, a signed graph $\dot{G}$ is called strongly regular if it is neither homogeneous complete nor edgeless, and there exist $r\in \mathbb{N}$, $a, b ,c\in\mathbb{Z}$, such that the entries of $A^2_{\dot{G}}$ satisfy

\begin{equation*}
  a^{(2)}_{ij}=
\left\{
\begin{aligned}
r, & \quad \text{if $v_i=v_j$}; \\
a, & \quad \text{if $v_i\overset{+}{\sim}v_j$}; \\
b, & \quad \text{if $v_i\overset{-}{\sim}v_j$}; \\
c, & \quad \text{if $v_i\neq v_j$ and $v_i\nsim v_j$}.
\end{aligned}
\right.
\end{equation*}
By the definition, a signed graph $\dot{G}$ is strongly regular if and only if its adjacency matrix satisfies that
\begin{equation}\label{def}
  A^2_{\dot{G}}+\frac{b-a}{2}A_{\dot{G}}=\frac{a+b}{2}A_G+cA_{\overline{G}}+rI=\frac{a+b-2c}{2}A_G+cJ+(r-c)I.
\end{equation}
where $\overline{G}$ is the complement of $G$, and $I$, $J$ are the identity matrix and the all-1 matrix, respectively.

If the strongly regular signed graph $\dot{G}$ is $\rho$ net-regular, then the all-1 vector $\mathbf{j}$ is the eigenvector of  $\dot{G}$, $G$ and $\overline{G}$. So we have
\begin{equation}\label{abcrn}
  \rho^2+\frac{b-a}{2}\rho=\frac{a+b}{2}r+c(n-r-1)+r.
\end{equation}
by equation \eqref{def}.

In \cite{KS}, Koledin and Stani\'{c} partitioned all inhomogeneous strongly regular signed graphs (for short, we write SRSGs) into the following five disjoint classes:

$\mathcal{C}_1$: SRSGs with $a=-b$, which are either complete or non-complete with $c\neq0$.

$\mathcal{C}_2$: SRSGs with $a=-b$, which are non-complete with $c=0$.

$\mathcal{C}_3$: SRSGs with $a\neq-b$, which are either complete or non-complete with $c=\frac{a+b} {2}$ .

$\mathcal{C}_4$: SRSGs with  $a\neq-b$, which are non-complete with $c=0$.

$\mathcal{C}_5$: SRSGs with  $a\neq-b$, which are non-complete with $c\neq\frac{a+b} {2}$ and $c\neq0$.

There are some results about this kind of strongly regular signed graphs. Stani\'c  established the certain structural and spectral properties of such strongly regular signed graphs and characterized the bipartite SRSGs and the SRSGs with 4 eigenvalues \cite{S1}. In \cite{KS}, Koledin and Stani\'{c} studied the SRSGs in class $\mathcal{C}_3$.  An\dj eli\'{c},   Koledin and Stani\'c considered  relationships between net-regular, strongly regular and walk-regular signed graphs in \cite{AKS} and they determined the net-regular strongly regular signed graphs with degree at most 4 in \cite{AKS1}.

 In this paper, we concentrate on the connected net-regular strongly regular signed graphs with degree $5$.  In Section 2, we give some known results for strongly regular signed graphs.  And all net-regular SRSGs with degree 5 are presented in Section 3.

 It is worth noting that the positive and negative edges of signed graphs in figures are represented by solid and dashed lines, respectively, in this paper.

\section{Preliminaries}

The following two known results will be used in the sequel. The first one examines the relationship between parameters $a$ and $b$ in an inhomogeneous strongly regular signed graph and its negation. The second one characterizes the number of negative walks of length 2 between two vertices in a connected non-complete net-regular strongly regular signed graph.

\begin{lemma}\cite{KS}\label{le1}
 If $\dot{G}$ is a SRSG belonging to $\mathcal{C}_i$, $1\leq i\leq 5$, so is $-\dot{G}$. The parameters $a$ and $b$ of an inhomogeneous SRSG $\dot{G}$ interchange their roles in the set of parameters of $-\dot{G}$.
\end{lemma}

\begin{lemma}\cite{AKS}\label{le2}
  Let $\dot{G}$ be a connected non-complete net-regular SRSG belonging to $\mathcal{C}_1\cup \mathcal{C}_4\cup \mathcal{C}_5$. Then for any two vertices $v_i$ and $v_j$ of $\dot{G}$, the number of negative walks of length 2 between $v_i$ and $v_j$ is even. Moreover, if this number is $2k$, $k\in \mathbb{N}$, then there are exactly $k$ common neighbours of $v_i$ and $v_j$ that are joined to $v_i$ by a positive edge and to $v_j$ by a negative edge, and exactly $k$ common neighbours that are joined to $v_i$ by a negative edge and to $v_j$ by a positive edge.
\end{lemma}

\section{The net-regular SRSGs with degree 5}

 In this section, we consider the net-regular SRSGs with degree 5. And we consider the positive net-degree by Lemma \ref{le1}, then there are two  possibilities for the net-degree: 3 and 1.

\subsection{3 net-regular SRSGs with degree 5}

Let $\dot{G}$ be a connected 5-regular and 3 net-regular SRSG, then  $G^+$ is 4-regular and $G^-$ is 1-regular. Therefore, $G^-$ induces a perfect matching of $\dot{G}$ and then the order of $\dot{G}$ is even.

If $\dot{G}\in \mathcal{C}_2$ or $\dot{G}$ is a complete signed graph in $\mathcal{C}_1$, then we have $A^2_{\dot{G}}+bA_{\dot{G}}=5I$ by equation \eqref{def}. It deduces that $\dot{G}$ has two eigenvalues: net-degree 3 and $\frac{-5}{3}$. This contradicts that $b$ is an integer.

We have known from \cite{KS} that every SRSG of $\mathcal{C}_3$ with degree $r\leq6$ must be complete. Then $G$ is isomorphic to $K_6$ if $\dot{G}\in \mathcal{C}_3$. It is easy to see that $\dot{G}$ is  isomorphic to the signed complete graph $\dot{G_1}$ (see in Figure \ref{G1}) with parameters  $(6,5,0,4)$.
\begin{figure}[h]
  \centering
  \includegraphics[width=5cm]{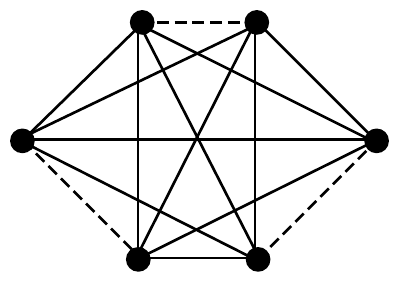}
    \caption{$\dot{G_1}$}\label{G1}
\end{figure}

Now, we need to consider the connected non-complete SRSGs in $\mathcal{C}_1\bigcup\mathcal{C}_4\bigcup\mathcal{C}_5$. According to equation \eqref{abcrn}, we have

\begin{equation}\label{abcn}
  -4a-b=c(n-6)-4.
\end{equation}

We concern that $\dot{G}$ contains an unbalanced triangle firstly. Since the vertex negative degree is 1, the  unbalanced triangle must have only one negative edge.
\begin{lemma}\label{UBT-5-3}
Let $\dot{G}$ be a connected non-complete 5-regular and 3 net-regular SRSG in  $\mathcal{C} _1\bigcup\mathcal{C}_4\bigcup\mathcal{C}_5$. If $\dot{G}$ contains an unbalanced triangle, then $\dot{G}$ is  isomorphic to $\dot{S}^1_8$ with parameters  $(8,5,-2,4,4)$ or $\dot{S}^1_{10}$  with parameters $(10,5,-2,4,2)$.
\end{lemma}

\begin{proof}
Suppose that vertices $v_i,v_j,v_k$, with $v_i\overset{+}{\sim}v_j,v_k$ and $v_j\overset{-}{\sim}v_k$, form an unbalanced triangle in  $\dot{G}$. By Lemma \ref{le2}, there is a common neighbour $v_l$  of $v_i$ and $v_j$ such that  $v_l\overset{-} {\sim}v_i$ and $v_l\overset{+}{\sim}v_j$. We consider edge $v_iv_k$, then there is also one common neighbour of $v_i$ and $v_k$  such that it is negatively adjacent to  $v_i$ and positively adjacent to $v_k$. Then we have $v_k\overset{+}{\sim}v_l$ since $v_i$ has only one negative neighbour. Suppose $v_m$ and $v_s$ are the remaining positive neighbours of $v_i$. Since there are two negative walks of length 2 between $v_i$ and $v_j$, the parameter $a\in\{ 0, -1, -2\}$ by $d^{-}(v_i)=d^{-}(v_j)=1$.

If $a=0$, then $v_m\overset{+}{\sim}v_j$ and $v_s\overset{+}{\sim}v_j$. Consider the positive edge $v_iv_k$,  we arrive at $v_k\overset{+}{\sim}v_m,v_s$.  And we have $v_m\overset{+}{\sim}v_l$, $v_m\overset{-}{\sim}v_s$ and $v_s\overset{+}{\sim}v_l$ by considering the positive edges $v_iv_m$ and $v_iv_s$ in turn. Now the underlying graph is a complete graph. A contradiction. So $a\neq 0$.

If $a=-1$, then one of $v_m,v_s$ is positively adjacent to $v_j$. We suppose that $v_m\overset{+}{\sim}v_j$ without losing generality.  Concern the positive edge $v_iv_k$, then we get that one of $v_m,v_s$ is positively adjacent to $v_k$. If $v_k\overset{+}{\sim}v_m$, there will be two positive walks of length 2 between $v_i$ and $v_m$. A contradiction. Therefore, $v_k\overset{+} {\sim}v_s$. Then we get that $v_m\overset{+}{\sim}v_l$ and $v_m\overset{-}{\sim}v_s$ by considering the positive edge $v_iv_m$. But there are two positive walks of length 2 between  $v_m$ and $v_j$, which contradicts $a=-1$. Therefore, $a\neq -1$.

If $a=-2$, then $v_m$ and $v_s$ are not adjacent to $v_j$.  Consider the positive edge $v_iv_m$, we obtain that $v_m\overset{+}{\sim}v_l$ and $v_m\overset{-}{\sim}v_s$. And we have $v_s\overset{+}{\sim}v_l$ for the positive edge $v_iv_s$. In this situation, we get $b=4$ by the negative edge $v_iv_l$. To satisfy the condition $b=4$, there must be two common neighbours of $v_k$ and $v_j$ which are positively adjacent to $v_k$ and $v_j$ and they are negatively adjacent to each other by $a=-2$. So do the vertices $v_m$ and $v_s$. Then there is a basic structure in $\dot{G}$ as shown in Figure \ref{b=4}. If $a=-2$ and $b=4$, we have $c(n-6)=8$ by equation \eqref{abcn}. So the possible parameters are $(8,5,-2,4,4)$, $(10,5,-2,4,2)$ and $(14,5,-2,4,1)$. As shown in Figure \ref{8,10}, we get the corresponding signed graphs of order 8, 10 and 14 by the  basic structure.   It is easy to see that $\dot{S}_{14}$ is not a SRSG and $\dot{S}^1_8$, $\dot{S}^1_{10}$ are SRSGs with parameters  $(8,5,-2,4,4)$ and $(10,5,-2,4,2)$, respectively.
\end{proof}
\begin{figure}
  \centering
  \includegraphics[width=5cm]{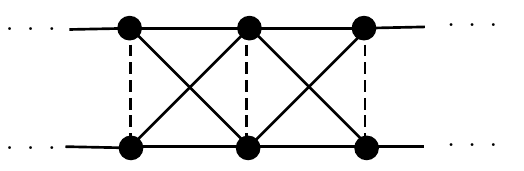}
  \caption{The basic structure of $(a,b)=(-2,4)$ in proof of Lemma \ref{UBT-5-3}}\label{b=4}
\end{figure}
\begin{figure}
  \centering
    \subfigure[$\dot{S}^1_8$]{
  \includegraphics[width=5.2CM]{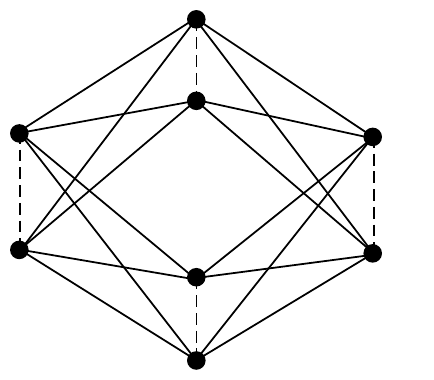}}
    \subfigure[ $\dot{S}^1_{10}$]{
  \includegraphics[width=4.5CM]{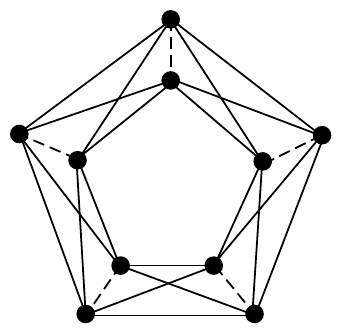}}
  \subfigure[ $\dot{S}^1_{14}$]{
  \includegraphics[width=4.5CM]{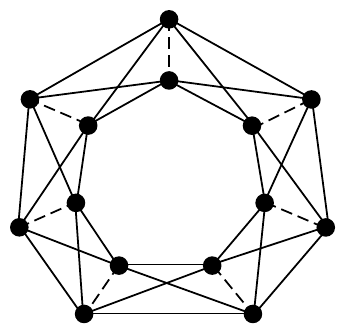}}
  \centering
  \caption{$\dot{S}^1_8$, $\dot{S}^1_{10}$ and  $\dot{S}_{14}$}\label{8,10}
\end{figure}

Now we consider that $\dot{G}$ does not contain  unbalanced triangles. Since $d^-(v)=1$ for every  $v\in V(\dot{G})$, we have $b=0$ by the Lemma \ref{le2} in this case. And the possible values of $a$ are  $0,1,2,3$.

We have $c(6-n)=4(a-1)$ by equation \eqref{abcn} and $b=0$. Then we get that the possible parameters are:
(1) $(10,5,0,0,1)$,  (2) $(8,5,0,0,2)$,  (3) $(n,5,1,0,0)$,  (4) $(10,5,2,0,-1)$,  (5) $(8,5,2,0,-2)$,  (6) $(14,5,3,0,-1)$,   (7) $(10,5,3,0,-2)$,  (8) $(8,5,3,0,-4)$
according to $n$ is even and $\dot{G}$ is non-complete and 5-regular.
%

Since $\dot{G}$ does not contain unbalanced triangles and $d^-(v)=1$ for every $v\in V(\dot{G})$, $a=0$ implies that two positively adjacent vertices in  $\dot{G}$ don't have  common neighbours. Therefore, two adjacent vertices in the underlying graph of $\dot{G}$ with parameters $(10,5,0,0,1)$ don't have common neighbours and two non-adjacent vertices have 1 or 5 common neighbours. There are sixty 5-regular graphs with $n=10$ (Figure \ref{10-5}) and only one graph, i.e. $K_{5,5}$, satisfies these two conditions. Then we can easily get that the corresponding SRSG is the signed graph $\dot{S}^2_{10}$ (Figure \ref{S10-1}). And two non-adjacent vertices in the underlying graph of  $\dot{G}$ with parameters $(10,5,2,0,-1)$ have three common neighbours.  But all of the 5-regular graphs with $n=10$ cannot satisfy this condition through simple verification.
If  $\dot{G}$ has parameters $(8,5,2,0,-2)$, then there exist two adjacent vertices in the underlying graph have no common neighbours.  There are three 5-regular graph with $n=8$ (Figure \ref{8-5}),  it is easy to see that none of them satisfies this condition.

\begin{figure}
  \centering
  \includegraphics[width=6cm]{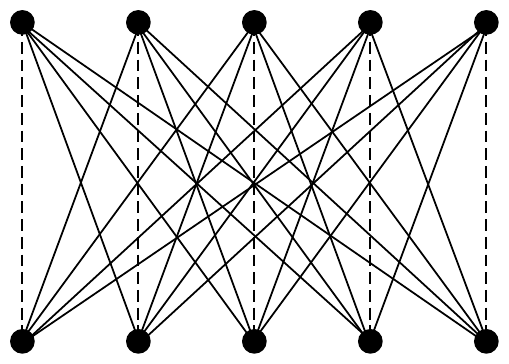}
  \caption{$\dot{S}^2_{10}$}\label{S10-1}
\end{figure}

\begin{figure}[h]
  \centering
  \subfigure{\includegraphics[width=3.3cm]{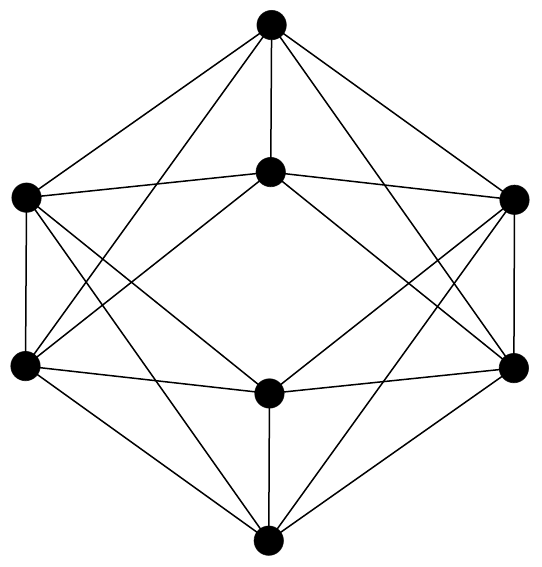}}
  \subfigure{\includegraphics[width=3.3cm]{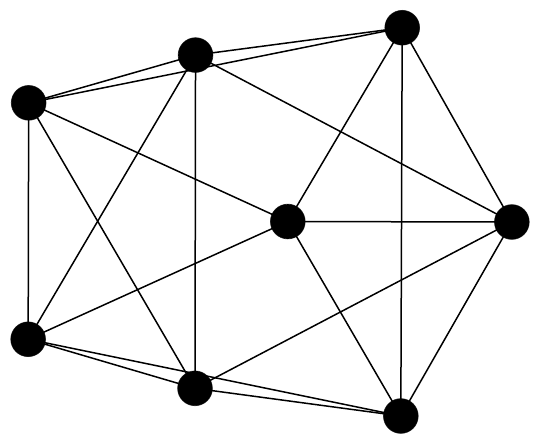}}
  \subfigure{\includegraphics[width=3.3cm]{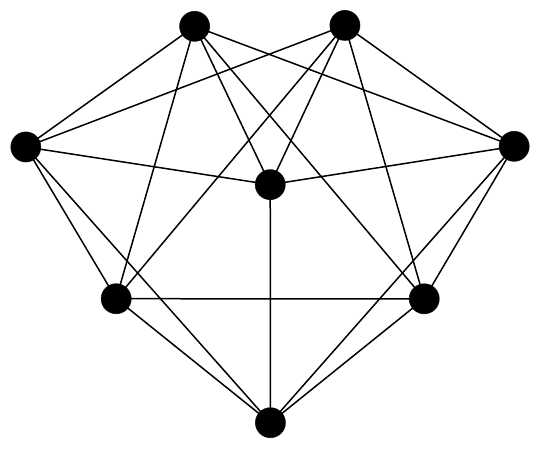}}
  \caption{All 5-regular graphs with order 8}\label{8-5}
\end{figure}
\begin{remark}
The data of the regular graph is obtained from the following website and \cite{M}:

http://www.mathe2.uni-bayreuth.de/markus/reggraphs.html\#CRG
\end{remark}

The underlying graph of  $\dot{G}$ with parameters $(8,5,0,0,2)$ is a strongly regular graph with parameters (8,5,0,2). Such strongly regular graph does not exist since it is contradictory to equation \eqref{srg}.

 If $a=3$, then $G^+=tK_5$ and so we have  $5|n$. This eliminates the parameter sets $(14,5,3,0,-1)$ and $(8,5,3,0,-4)$. For parameters  $(10,5,3,0,-2)$, we have $G^+=2K_5$ by $n=10$. Then we can easily get that the corresponding SRSG $\dot{S}^3_{10}$ (Figure \ref{S10}) by $b=0$ and $c=-2$.

 \begin{figure}[htbp]
  \centering
  \includegraphics[width=7cm]{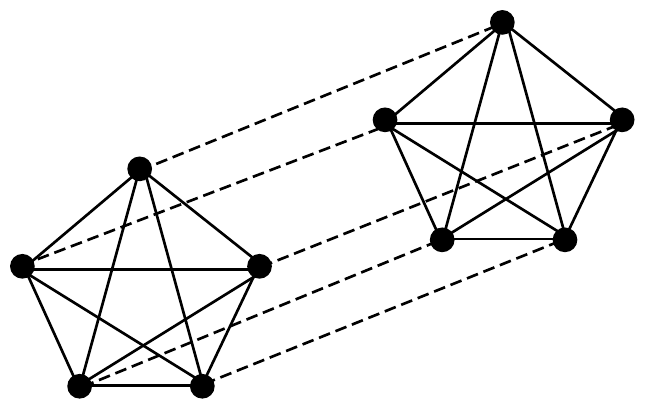}
  \caption{$\dot{S^3_{10}}$}\label{S10}
\end{figure}

There are two cases for $c=0$. The first one is that there are no common neighbours between two non-adjacent vertices. The second one is that there are two positive  and two negative walks of length 2 between two non-adjacent vertices.   Let $\dot{G}$ be a connected non-complete 5-regular and 3 net-regular SRSG with parameters $(n,5,1,0,0)$. Suppose $v_iv_j$ is a negative edge of $\dot{G}$ and $v_k,v_l,v_m,v_s$ are the remaining positive neighbours of $v_i$. Since $a=1$ and $b=0$,  two positively adjacent vertices have only one positive walk of length 2 which consists of two positive edges. So we suppose $v_k\overset{+}{\sim}v_l$ and  $v_m\overset{+}{\sim}v_s$ without loss of generality. Let $v_{k_1},v_{l_1},v_{m_1},v_{s_1}$ be positive neighbours of $v_j$ and suppose that $v_{k_1}\overset{+}{\sim}v_{l_1}$ and  $v_{m_1}\overset{+}{\sim}v_{s_1}$. For two non-adjacent vertices $v_k$ and $v_j$, there are two positive and two negative walks of length 2 between them since they have a common neighbour $v_i$. Without loss of generality, we suppose that $v_k\overset{-}{\sim}v_{k_1}$. Then $v_k\nsim v_{l_1}$ since $b=0$. Therefore, we have $v_k\overset{+}{\sim}v_{m_1}$ and $v_k\overset{+}{\sim}v_{s_1}$. But there are two positive walks of length 2 between $v_{m_1}$ and $v_{s_1}$. A contradiction.

\begin{theorem}
  There are five connected 5-regular and 3 net-regular SRSGs: $\dot{G}_1$,   $\dot{S}^1_{8}$,  $\dot{S}^1_{10}$, $\dot{S}^2_{10}$ and $\dot{S}^3_{10}$.
\end{theorem}

\subsection{ 1 net-regular SRSGs with degree 5}

Let $\dot{G}$ be a connected 5-regular and 1 net-regular SRSG, then we have $G^+$ is 3-regular and $G^-$ is 2-regular. And the net-degree 1 is an eigenvalue of $\dot{G}$.

If $\dot{G}\in\mathcal{C}_2$, then $a=-b$ and so $A^2_{\dot{G}}+bA_{\dot{G}}=5I$. Therefore, the eigenvalues of $\dot{G}$ are 1 and $-5$. Then $\dot{G}$ is switching equivalent to $-K_6$ by the Theorem 3.1 in \cite{S2}. But this contradicts that $\dot{G}$ is non-complete.

If $\dot{G}$ is complete, then $G^-$ is $C_6$ or $2C_3$.  It is easy to get that there is only one strongly regular signed graph $\dot{G_2}$ (Figure \ref{G2}) which has parameters $(6,5,-4,4)$.

\begin{figure}[h]
  \centering
    \includegraphics[width=4CM]{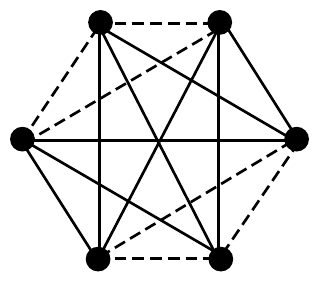}
  \caption{$\dot{G_2}$}\label{G2}
\end{figure}

 The remaining cases are the connected and non-complete net-regular SRSGs in $\mathcal{C} _1\bigcup\mathcal{C}_4\bigcup\mathcal{C}_5$. Firstly, we give the following conditions for parameters $a$ and $b$.

\begin{proposition}\label{5-1}
Let $\dot{G}\in \mathcal{C} _1\bigcup\mathcal{C}_4\bigcup\mathcal{C}_5$ be a connected and non-complete 5-regular and 1 net-regular SRSG. Then we have the following statements:

 (1) $a<3$;

 (2) If $a=0$ ($b=0$, respectively), then  two positively (negatively, respectively) adjacent vertices don't have common neighbours;

 (3) If $a=-1$, then $b\leq0$;

 (4) If $a=-2$, then $b\geq-1$;

 (5) $a>-3$;

 (6) $-2\leq b<3$;

 (7) If there is a balanced triangle with two negative edges, then $b\leq0$;

 (8) If $b=2$ or $-1$, then $a\leq0$;

 (9) If $b=-2$, then $a\geq1$.

\end{proposition}

\begin{proof}
  (1) Suppose that $v_iv_j$ is a positive edge of $\dot{G}$. If $a=4$, then $v_i,v_j$ have four common
neighbours. We suppose that $N(v_i)\bigcap N(v_j)=\{v_k,v_l,v_m,v_s\}$  and $v_k\overset{+} {\sim}v_i,v_j$,\; $v_l\overset{+} {\sim}v_i,v_j$,\; $v_m\overset{-} {\sim}v_i,v_j$ and $v_s\overset{-} {\sim}v_i,v_j$. Considering the positive edge $v_iv_k$, we have $v_k\overset{-} {\sim}v_m,v_s$. But this is contradictory to the net-degree of $v_m$ and $v_s$.

  If $a=3$, then there are three positive walks of length 2 between $v_i$ and $v_j$. We suppose that $N(v_i)\bigcap N(v_j)=\{v_k,v_l,v_m\}$  and $v_k\overset{+} {\sim}v_i,v_j$,\; $v_l\overset{+} {\sim}v_i,v_j$,\; $v_m\overset{-} {\sim}v_i,v_j$ or $v_k\overset{+} {\sim}v_i,v_j$,\; $v_l\overset{-} {\sim}v_i,v_j$,\; $v_m\overset{-} {\sim}v_i,v_j$.

  Case 1. $v_k,v_l\overset{+} {\sim}v_i,v_j$,\; $v_m\overset{-} {\sim}v_i,v_j$. Let $v_{s_1}$ and $v_{s_2}$ be the remaining negative neighbours of $v_i$ and $v_j$, respectively. Concern the positive edge $v_iv_k$, then we have $v_k\overset{+} {\sim}v_l$ and $v_k\overset{-} {\sim}v_{s_1}$ since $d^-(v_m)=2$. We also have $v_k\overset{-} {\sim}v_{s_2}$ by considering the positive edge $v_jv_k$. In a similar way, we get that $v_l\overset{-} {\sim}v_{s_1}, v_{s_2}$. But this is contradictory to the net-degree of $v_{s_1}$ and $v_{s_2}$.

  Case 2. $v_k\overset{+} {\sim}v_i,v_j$, $v_m,v_l\overset{-} {\sim}v_i,v_j$.  For the positive edge $v_iv_k$, we must have $v_k\overset{-} {\sim}v_l$ or $v_k\overset{-} {\sim}v_m$. But this is contradictory to the net-degree of $v_l$ or $v_m$.

  (2)  Suppose that $v_iv_j$ is a positive edge. If $a=0$ and  $v_i$ have  common neighbours with $v_j$, then $|N(v_i)\bigcap N(v_j)|=4$. Supposes that $N(v_i)\bigcap N(v_j)=\{v_k,v_l,v_m,v_s\}$ and $v_k\overset{+} {\sim}v_i,v_j$,\; $v_l\overset{-} {\sim}v_i, v_l\overset{+} {\sim}v_j$,\; $v_m\overset{+} {\sim}v_i, v_m\overset{-} {\sim}v_j$,\; $v_s\overset{-} {\sim}v_i,v_j$. Considering the positive edge $v_iv_k$, we have $v_k\overset{-} {\sim}v_l,v_m$ and $v_k\overset{+} {\sim}v_s$. But the situation of walks of length 2 between $v_i$ and $v_m$ contradicts Lemma \ref{le2}. So $N(v_i)\bigcap N(v_j)=\emptyset$ .

  Suppose that $v_iv_j$ is a negative edge.  If $b=0$ and $v_i$ has common neighbours with $v_j$, then $|N(v_i)\bigcap N(v_j)|=4$. We suppose that  $N(v_i)\bigcap N(v_j)=\{v_k,v_l,v_m,v_s\}$ and $v_k\overset{+} {\sim}v_i,v_j$,\; $v_l\overset{+} {\sim}v_i, v_l\overset{-} {\sim}v_j$,\; $v_m\overset{-} {\sim}v_i, v_m\overset{+} {\sim}v_j$,\; $v_s\overset{+} {\sim}v_i,v_j$. There is one negative walk of length 2 between $v_i$ and $v_m$ for the negative edge $v_iv_m$. So we have three possibilities of the remaining negative neighbour of $v_m$: $v_m\overset{-} {\sim}v_l$ or  $v_m\overset{-} {\sim}v_k$ or  $v_m\overset{-} {\sim}v_t$.

   Case 1. $v_m\overset{-} {\sim}v_l$. Then we have $v_m\overset{+} {\sim}v_k,v_s$. For the negative edge $v_jv_l$, there are two negative walks of length 2 between $v_j$ and $v_l$. So we get that $v_l\overset{+} {\sim}v_k,v_s$.  But this is  contradictory to $d^{+}(v_k)$ and $d^{+}(v_s)$.

   Case 2. $v_m\overset{-} {\sim}v_k$. Then we have $v_m\overset{+} {\sim}v_l,v_s$. Consider the negative edge $v_jv_l$, then we get that $v_l\overset{-} {\sim}v_k$ or $v_l\overset{-} {\sim}v_s$. If $v_l\overset{-} {\sim}v_k$,  then $v_l\overset{+} {\sim}v_s$ and so $d^{+}(v_s)=4$. A contradiction. If  $v_l\overset{-} {\sim}v_s$, then $v_l\overset{+} {\sim}v_k$. And we obtain that $v_k\overset{-} {\sim}v_s$ by considering the negative edge $v_kv_m$. But $\dot{G}$ is complete now. A contradiction.

   Case 3. $v_m\overset{-} {\sim}v_s$.   This case is similar as  Case 2. We omit it.

  (3) Suppose that $v_iv_j$  is a  positive edge of $\dot{G}$ and $a=-1$. Then there are two negative walks and one positive walk of length 2 between two positively adjacent vertices. Therefore,  $v_i$ and $v_j$ have three common neighbours, say $v_k,v_l,v_m$, and there are two cases:

  Case 1. $v_k\overset{+} {\sim}v_i$, $v_k\overset{-} {\sim}v_j$, $v_l\overset{-} {\sim}v_i$, $v_l\overset{+} {\sim}v_j$, $v_m\overset{+} {\sim}v_i,v_j$. In this case we have $v_k\overset{-} {\sim}v_m$ by positive edge $v_iv_m$. But the situation of two negative walks of length 2 between $v_i$ and $v_k$ is  contradictory to Lemma \ref{le2}.

  Case 2. $v_k\overset{+} {\sim}v_i$, $v_k\overset{-} {\sim}v_j$, $v_l\overset{-} {\sim}v_i$, $v_l\overset{+} {\sim}v_j$, $v_m\overset{-} {\sim}v_i,v_j$.  There is one negative walk of length 2 between $v_i$ and $v_m$, then they must have another one negative walk of length 2 by Lemma \ref{le2}. So we have $b\leq0$.

  (4) Let $v_iv_j$ be a  positive edge of $\dot{G}$ and $a=-2$. Then there are two negative walks of length 2 between $v_i$ and $v_j$. Suppose that $v_k\in N(v_i)\bigcap N(v_j)$ and $v_k\overset{-} {\sim}v_i$, $v_k\overset{+} {\sim}v_j$. Then there is one positive walk of length 2 between $v_i$ and $v_k$. So $b\geq-1$.

  (5) Let $v_iv_j$ be a  positive edge of $\dot{G}$, then there are at most four negative walks of length 2 between $v_i$ and $v_j$ by the vertex net-degree and Lemma \ref{le2}.  If $a=-4$, then there are four negative walks of length 2 between $v_i$ and $v_j$.  Suppose $N(v_i)\bigcap N(v_j)=\{v_k,v_l,v_m,v_s\}$ and  $v_k,v_l\overset{+} {\sim}v_i$,\; $v_k,v_l\overset{-} {\sim}v_j$,\; $v_m,v_s\overset{-} {\sim}v_i$ and $v_m,v_s\overset{+} {\sim}v_i$. We can get that $\dot{G}$ is complete by considering the positive edges  $v_iv_k,v_iv_l,v_jv_s$ in turn. A contradiction.   And we have $a\neq-3$ by the Lemma \ref{le2}. Then $a>-3$.

  (6) Suppose that $v_iv_j$ is a negative edge of $\dot{G}$. If $b=4$, then there are four positive walks of length 2 between $v_i$ and $v_j$. Suppose $N(v_i)\bigcap N(v_j)=\{v_k,v_l,v_m,v_s\}$ and $v_k\overset{+} {\sim}v_i,v_j$,\; $v_l\overset{+} {\sim}v_i,v_j$,\; $v_m\overset{-} {\sim}v_i,v_j$,\; $v_s\overset{+} {\sim}v_i,v_j$. Considering the negative edge $v_iv_m$, we have $v_m\overset{+} {\sim}v_k, v_l,v_s$. Now, there are two negative walks of length 2 between two  positively adjacent vertices. And by the situation of these two negative walks, we get that there are the other two negative walks of length 2 between two  positively adjacent vertices. Then $a=-4$, which is  contradictory to (5).

  If $b=3$, then  there are three positive walks of length 2 between $v_i$ and $v_j$. Suppose that $N(v_i)\bigcap N(v_j)=\{v_k,v_l,v_m\}$  and $v_k\overset{+} {\sim}v_i,v_j$,\; $v_l\overset{+} {\sim}v_i,v_j$,\; $v_m\overset{+} {\sim}v_i,v_j$ or $v_k\overset{+} {\sim}v_i,v_j$,\; $v_l\overset{+} {\sim}v_i,v_j$,\; $v_m\overset{-} {\sim}v_i,v_j$. Let $v_{s_1},v_{s_2}$ be the remaining neighbours of $v_i$ and $v_j$, respectively.

  For the first case, we have   $v_{s_1}\overset{-} {\sim}v_i$ and $v_{s_2}\overset{-} {\sim}v_j$. Then $v_{s_1}\overset{+} {\sim}v_k,v_l,v_m$ and $v_{s_2}\overset{+} {\sim}v_k,v_l,v_m$ by $b=3$. But we get that $d^+(v_k)=d^+(v_l)=d^+(v_m)=4$. A contradiction.

  For the second case, we have  $v_{s_1}\overset{+} {\sim}v_i$ and $v_{s_2}\overset{+} {\sim}v_j$. And $v_m$ must be positively adjacent to one of $v_k,v_l$ by negative edge $v_iv_m$. Suppose that $v_m\overset{+} {\sim}v_k$ without loss of generality,  then there are two negative walks of length 2 between $v_i$ and $v_k$. But the situation of these two negative walks deduces that  $a=-4$.   This contradicts (5).

  Therefore, we get that $b<3$. And it is obvious that there are at most two negative walks of length 2 between $v_i$ and $v_j$. So $b\geq-2$.

  (7) Suppose that $v_iv_jv_k$ forms a triangle in $\dot{G}$ such that $v_i\overset{+} {\sim}v_j$ and $v_k\overset{-} {\sim}v_i,v_j$. Since there is one negative walk of length 2 between $v_k$ and $v_i$, then there must be another negative walk of length 2 between them by Lemma \ref{le2}. So we have $b\leq0$.

  (8) If $b\in\{2,-1\}$, there must be an unbalanced triangle in $\dot{G}$. Therefore, we obtain that $a\leq0$ by Lemma \ref{le2}.

  (9) Suppose  $v_iv_j$ is a negative edge of $\dot{G}$. If $b=-2$, then there are two negative walks of length 2 between $v_i$ and $v_j$. Let $v_k,v_l$ be two common neighbours of $v_i$ and $v_j$ such that $v_k\overset{+} {\sim}v_i$, $v_k\overset{-} {\sim}v_j$, $v_l\overset{-} {\sim}v_i$ and $v_l\overset{+} {\sim}v_j$. We claim that  $v_k,v_l$ can not be joined by a positive edge. Otherwise there will be one positive walk of length 2 between $v_i$ and $v_l$, which is  contradictory to $b=-2$. So there is no negative walk of length 2 and at least one positive walk of length 2 between $v_i$ and $v_k$. Therefore, $a\geq1$.

\end{proof}

If $\dot{G}\in \mathcal{C} _1\bigcup\mathcal{C}_4\bigcup\mathcal{C}_5$ is a connected and non-complete 5-regular and 1 net-regular SRSG with parameters  $(n,r,a,b,c)$, the possible values of $(a,b)$ are:
$ (2,1)$, $(2,0)$, $(2,-2)$, $(1,1)$, $(1,0)$, $(1,-2)$, $(0,2)$, $(0,1)$, $(0,0)$, $(0,-1)$, $(-1,0)$, $(-1,-1)$, $(-2,2)$, $(-2,1)$, $(-2,0)$, $(-2,-1)$
by the above Lemma. And the equation \eqref{abcrn} will be

\begin{equation}\label{5-1-abcn}
  c(6-n)=4+3a+2b.
\end{equation}

Now, we concern the above possible values of $(a,b)$ one by one.

\begin{lemma}
Let $\dot{G}\in \mathcal{C} _1\bigcup\mathcal{C}_4\bigcup\mathcal{C}_5$ be a connected and non-complete 5-regular and 1 net-regular SRSG. If $(a,b)=(2,1)$, then $\dot{G}$ is the signed graph $\dot{S^1_{12}}$ (Figure \ref{S12}) with parameters $(12,5,2,1,-2)$.
\end{lemma}
\begin{proof}
   Suppose that $(a,b)=(2,1)$ and $v_iv_j$ is a positive edge of $\dot{G}$, then the edges of two positive walks of length 2 between $v_i$ and $v_j$ are positive by Proposition \ref{5-1}. Let $v_k$ and $v_l$ be  two common neighbours of $v_i$ and $v_j$ and then $v_k,v_l\overset{+}{\sim}v_i,v_j$. Suppose $v_m,v_t$ are the remaining neighbours of $v_i$ which are negatively adjacent to $v_i$.  Since  $(a,b)=(2,1)$, the only possibility is that $v_k\overset{+}{\sim}v_l$ and $v_m\overset{-}{\sim}v_t$. So $G^+$ is  isomorphic to $pK_4$ ($p\geq1$) and $G^-$ is  isomorphic to $qK_3$ ($q\geq1$). Therefore, $12|n$.

  Since $a=2$ and $b=1$, we have $c(6-n)=12$ by equation \eqref{5-1-abcn}. So the corresponding parameters sets are $(18,5,2,1,-1)$, $(12,5,2,1,-2)$, $(10,5,2,1,-3)$, $(9,5,2,1,-4)$ and $(8,5,2,1,-6)$. We only need to consider $(12,5,2,1,-2)$ by $12|n$ and then $p=3$ and $q=4$. Suppose vertices $v_1,v_2,v_3,v_4$ form a  $+K_4$ in $\dot{G}$, so do vertices $v_5,v_6,v_7,v_8$ and $v_9,v_{10},v_{11},v_{12}$.  And we suppose $v_1\overset{-}{\sim}v_5$. Then they must have one common neighbour which belongs to  $\{v_9,v_{10},v_{11},v_{12}\}$. Without loss of generality, we suppose $v_9\overset{-}{\sim}v_1,v_5$. By the same way, we have $v_2,v_6,v_{10}$ and $v_3,v_7,v_{11}$ and $v_4,v_8,v_{12}$ form triangles, respectively. Then $\dot{G}$ is isomorphic to the signed graph $\dot{S^1_{12}}$ as shown in Figure \ref{S12}.
\end{proof}

\begin{figure}[h]
  \centering
  \includegraphics[width=8cm]{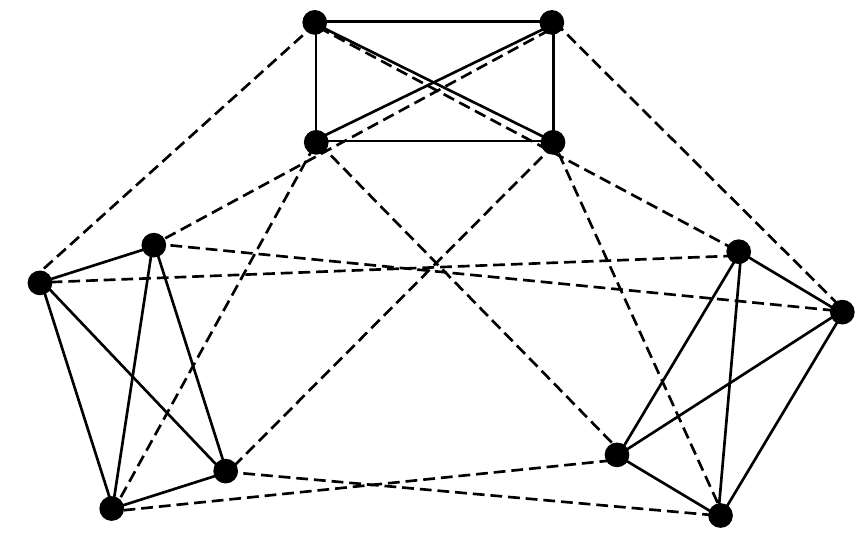}
  \caption{$\dot{S^1_{12}}$}\label{S12}
\end{figure}

\begin{lemma}
Let $\dot{G}\in \mathcal{C} _1\bigcup\mathcal{C}_4\bigcup\mathcal{C}_5$ be a connected and non-complete 5-regular and 1 net-regular SRSG. If $(a,b)=(2,0)$, then $4|n$.
\end{lemma}
\begin{proof}
Suppose $v_iv_j$ is a positive edge of $\dot{G}$, then all the edges of two positive walks of length 2 between  $v_i$ and $v_j$ are positive since $b=0$. Let $v_k,v_l$ be the common neighbours of $v_i$ and $v_j$, then $v_k\overset{+}{\sim}v_l$. Thus, $G^+$ is the copies of $K_4$ and so $4|n$.
\end{proof}

If $a=2, b=0$, we have $c(6-n)=10$ by equation \eqref{5-1-abcn}. Then the possible parameters sets are $(16,5,2,0,-1)$, $(11,5,2,0,-2)$ and $(8,5,2,0,-5)$. The last one is not true since the $c=-5$ is  contradictory to Lemma \ref{le2} and $r=5$.  And the case of $n=11$ is eliminated by  $4|n$. So we only need to consider $(16,5,2,0,-1)$.

 Since $c=-1$, there are two negative walks and one positive walk of length 2 between two non-adjacent vertices. And the edges of this positive walk of length 2 are negative since $G^+$ is the copies of $K_4$. Let $v_iv_j$ be a positive edge of $\dot{G}$ and $v_k,v_l$ be the common neighbours of $v_i$ and $v_j$. Suppose $v_m,v_s$ are negative neighbours of $v_i$ and $v_{m_1},v_{s_1}$ are negative neighbours of $v_j$. Since vertices $v_i$ and $v_{m_1}$ are non-adjacent and there is one negative walk of length 2 ($v_iv_jv_{m_1}$) between them, we have $v_m\overset{+}{\sim}v_{m_1}$ and $v_s\overset{-}{\sim}v_{m_1}$ (or $v_m\overset{-}{\sim}v_{m_1}$ and $v_s\overset{+}{\sim}v_{m_1}$). In the same way, we have $v_m\overset{-}{\sim}v_{s_1}$ (or $v_s\overset{-}{\sim}v_{s_1}$) for  $v_j$ and $v_m$ (or $v_s$). Now, $v_iv_mv_{s_1}v_jv_{m_1}v_sv_i$ (or  $v_iv_mv_{m_1}v_jv_{s_1}v_sv_i$) is a cycle of length 6. Next, we consider the negative neighbours of $v_k$ and $v_l$. Suppose $v_t,v_w$ are negative neighbours of $v_k$ and $v_{t_1},v_{w_1}$ are negative neighbours of $v_l$. Vertices $v_k$ and $v_{w_1}$ are non-adjacent and there is one negative walk of length 2 ($v_kv_lv_{w_1}$) between them. Therefore, we have $v_w\overset{+}{\sim}v_{w_1}$ and $v_t\overset{-}{\sim}v_{w_1}$ (or $v_w\overset{-}{\sim}v_{w_1}$ and $v_t\overset{+}{\sim}v_{w_1}$). Also, we have $v_w\overset{-}{\sim}v_{t_1}$ (or $v_t\overset{-}{\sim}v_{t_1}$) since  $v_l$ and $v_w$ (or $v_t$) are non-adjacent. Now, $v_kv_wv_{t_1}v_lv_{w_1}v_tv_k$ (or  $v_kv_wv_{w_1}v_lv_{t_1}v_tv_k$) is a cycle of length 6.  Then we have  $6|n$ and so $n\neq16$.

If $a=2$ and $b=-2$, we have $c(6-n)=6$ by equation \eqref{5-1-abcn}. Then the possible parameters sets are $(12,5,2,-2,-1)$, $(9,5,2,-2,-2)$ and $(8,5,2,-2,-3)$. Since there is no 5-regular graph with odd order, we have $n\neq 9$. The underlying graphs of $\dot{G}$ with $(12,5,2,-2,-1)$ and $(8,5,2,-2,-3)$ are strongly regular graphs with parameters $(12,5,2,3)$ and $(8,5,2,5)$, respectively. The first one is not true by equation \eqref{srg}. The second one satisfies  equation \eqref{srg}. But all the 5-regular graphs with order 8 (Figure \ref{8-5}) are not strongly regular graphs. So there is no connected non-complete 5-regular and 1 net-regular SRSG with $a=2$ and $b=-2$.

If $a=b=1$ or $(a,b)=(1,-2)$,  we have $c(6-n)=9$ or $c(6-n)=3$ by equation \eqref{5-1-abcn} and so the order must be odd, which is contradictory to $r=5$.

\begin{lemma}
  If $\dot{G}\in \mathcal{C} _1\bigcup\mathcal{C}_4\bigcup\mathcal{C}_5$ is a connected non-complete 5-regular and 1 net-regular SRSG, then the parameters $(a,b)\neq(1,0)$.
\end{lemma}

\begin{proof}
  If $b=0$, then there is no common neighbour for two negatively adjacent vertices. Suppose $v_iv_j$ is a positive edge, then there is only one positive walk of length 2 between $v_i$ and $v_j$. Let $v_k$ be the common neighbour of  $v_i,v_j$ and $v_t$ be the remaining positive neighbour of $v_i$. Then $v_t\nsim v_j$ and $v_t\nsim v_k$ by $a=1$. Since $b=0$ and $d^+(v_i)=3$, there is no positive walk of length 2 between $v_i$ and $v_t$. This is  contradictory to $a=1$.
\end{proof}

\begin{lemma}
  If $\dot{G}\in \mathcal{C} _1\bigcup\mathcal{C}_4\bigcup\mathcal{C}_5$ is a connected and non-complete 5-regular and 1 net-regular SRSG, then the parameters $(a,b)\neq(0,2), (-2,0)$.
\end{lemma}

\begin{proof}
  If $(a,b)=(0,2)$, then two positively adjacent vertices don't have common neighbours. And there are two  positive walks of length 2 between two negatively adjacent vertices. Suppose $v_iv_j$ is a negative edge of $\dot{G}$ and $N(v_i)\bigcap N(v_j)=\{v_k,v_l\}$. Then there are two cases: $v_k,v_l\overset{+}{\sim}v_i,v_j$; or  $v_k\overset{+}{\sim}v_i,v_j$ and $v_l\overset{-}{\sim}v_i,v_j$. But these two cases are contradictory to $a=0$.

  The case of $(a,b)= (-2,0)$ is similar. There are two  negative walks of length 2 between two positively adjacent vertices, which is  contradictory to $b=0$.
\end{proof}

\begin{lemma}
  If $\dot{G}\in \mathcal{C} _1\bigcup\mathcal{C}_4\bigcup\mathcal{C}_5$ is a connected non-complete 5-regular and 1 net-regular SRSG with parameters $(n,r,a,b,c)$, then $(a,b)\neq(0,1)$.
\end{lemma}

\begin{proof}
  If $(a,b)=(0,1)$, then two positively adjacent vertices don't have common neighbours. And there is one  positive walk of length 2 between two negatively adjacent vertices. Suppose $v_iv_j$ is a negative edge of $\dot{G}$ and $N(v_i)\bigcap N(v_j)=\{v_k\}$. Then $v_k\overset{-}{\sim}v_i,v_j$ by $a=0$. Therefore, $G^-$ is the copies of $K_3$ and so $3|n$. Then we get the possible parameters sets are $(12,5,0,1,-1)$, $(9,5,0,1,-2)$ and $(8,5,0,1,-3)$ by  $(a,b)=(0,1)$ and equation \eqref{5-1-abcn}. The last one is not true by $3|n$. And since $a=0$, the positive neighbours of $v_i,v_j,v_k$ don't coincide, which deduces that $n\geq 12$. So we only need to consider $(12,5,0,1,-1)$. Then $G^-=4K_3$.

  Let $v_1v_2v_3$, $v_4v_5v_6$,  $v_7v_8v_9$ and  $v_{10}v_{11}v_{12}$ be four triangles of $\dot{G}$. Since the vertex positive degree is 3 and $a=0$, we suppose $v_1\overset{+}{\sim}v_4,v_7,v_{10}$,\;  $v_2\overset{+}{\sim}v_5,v_8,v_{11}$ and $v_3\overset{+}{\sim}v_6,v_9,v_{12}$ without loss of generality. Then $v_1\nsim v_5,v_6,v_8,v_9,v_{11},v_{12}$. For vertices $v_1$ and $v_5$, $v_5$ is positively adjacent to one of $\{v_7,v_{10}\}$ by $c=-1$ and $a=0$. We suppose that $v_5\overset{+}{\sim}v_7$ without loss of generality. Then we have $v_6\overset{+}{\sim}v_{10}$ for non-adjacent vertices $v_1$ and $v_6$. By the same way, we suppose $v_8\overset{+}{\sim}v_{4}$, $v_9\overset{+}{\sim}v_{10}$, $v_{11}\overset{+}{\sim}v_4$ and $v_{12}\overset{+}{\sim}v_7$.

  The non-adjacent vertices $v_2$ and $v_6$ also need one positive walks of length 2. So $v_6$ is positively adjacent to one of $v_8,v_{11}$. But this contradicts $a=0$ since $v_8,v_{11}\overset{+}{\sim}v_4$ and $v_4\overset{-}{\sim}v_6$.

  \end{proof}

For the cases of $(a,b)=(0,0), (0,-1), (-1,0), (-1,-1), (-2,2)$, the possible parameters sets are $(7,5,0,0,-4)$, $(8,5,0,0,-2)$, $(10,5,0,0,-1)$, $(8,5,0,-1,-1)$, $(7,5,0,-1,-2)$, $(7,5,-1,0,-1)$,   $(7,5,-1,-1,1)$,  $(8,5,-2,2,-1)$, $(7,5,-2,2,-2)$ by equation \eqref{5-1-abcn}. Since $n$ cannot be odd, we need to consider the parameters $(8,5,0,0,-2)$, $(10,5,0,0,-1)$, $(8,5,0,-1,-1)$ and $(8,5,-2,2,-1)$.  There are three 5-regular graphs of $n=8$ (Figure \ref{8-5}) and every pair of adjacent vertices in these three graphs have common neighbours, which is  contradictory to $a=0$. So  the sets of $(8,5,0,0,-2)$ and  $(8,5,0,-1,-1)$ are not true. There are sixty 5-regular graphs of $n=10$ (Figure \ref{10-5}).  Only the complete bipartite graph $K_{5,5}$ satisfies that two adjacent vertices don't have common neighbours. If $K_{5,5}$ is the underlying graph of the SRSG with parameters $(10,5,0,0,-1)$,  then there will be two positive and three negative walks of length 2 between two non-adjacent vertices. This is impossible by Lemma \ref{le2}. The underlying graph of a SRSG  with parameters $(8,5,-2,2,-1)$ is a strongly regular graph with parameters $(8,5,2,3)$. But there is no such strongly regular graph since it is  contradictory to equation \eqref{srg}.

\begin{lemma}
  If $\dot{G}\in \mathcal{C} _1\bigcup\mathcal{C}_4\bigcup\mathcal{C}_5$ is a connected and non-complete 5-regular and 1 net-regular SRSG, then the parameters $(a,b)\neq(-2,1), (-2,-1)$.
\end{lemma}

\begin{proof}
  Suppose $v_iv_j$ is an edge of $\dot{G}$ and $\sigma$ is the sign function. If $(a,b)=(-2,1)$ and $\sigma(v_iv_j)=1$, then there are two negative walks of length 2 between $v_i$ and $v_j$ and there is only one  positive walk of length 2 between two negatively adjacent vertices. We suppose $N(v_i)\bigcap N(v_j)=\{v_k,v_l\}$ where $v_k\overset{+}{\sim}v_i,v_k\overset{-}{\sim}v_j$ and $v_l\overset{-}{\sim}v_i,v_l\overset{+}{\sim}v_j$. Then we have $v_k\overset{+}{\nsim}v_l$. Otherwise, there will be two positive walks of length 2 between $v_k$ and $v_j$, which is contradictory to $b=1$. Let $v_m$ and $v_s$ be the remaining neighbours of $v_i$ such that  $v_m\overset{-}{\sim}v_i$ and $v_s\overset{+}{\sim}v_i$. By the positive edge  $v_iv_k$, we have $v_m\overset{+}{\sim}v_k$ since $v_k\overset{+}{\nsim}v_l$. Since the edges $v_iv_k$, $v_iv_l$ and $v_iv_m$ satisfy the $(a,b)=(-2,1)$ now, we have $v_s\nsim v_k,v_l, v_m$. But this is  contradictory to $a=-2$.

   If $(a,b)=(-2,-1)$ and $\sigma(v_iv_j)=-1$, then there are two negative walks and one positive walk of length 2 between $v_i$ and $v_j$.  We suppose that $N(v_i)\bigcap N(v_j)=\{v_k,v_l,v_m\}$  with $v_k\overset{+}{\sim}v_i,v_k\overset{-}{\sim}v_j$, $v_l\overset{-}{\sim}v_i,v_l\overset{+}{\sim}v_j$ and $v_m\overset{+}{\sim}v_i,v_j$. But there is one positive walk of length 2 between  $v_i$ and $v_k$, which also contradicts $a=-2$.
\end{proof}

The following theorem can be obtained by summing up the above contents.

\begin{theorem}
  There are two connected 5-regular and 1 net-regular SRSGs: $\dot{G}_2$ and $\dot{S}^1_{12}$
\end{theorem}

\vspace*{2mm}

\textbf{ Declaration of competing interest}

The authors declare that they have no conflict of interests.

\vskip 0.6 true cm
{\textbf{Acknowledgments}}

 This research is supported by the National Natural Science Foundation of China (Nos. 12331012,11971164).
\baselineskip=0.25in

\linespread{1.10}

\begin{appendices}
\section{}

\begin{figure}
  \centering
  \subfigure{
  \includegraphics[width=12cm]{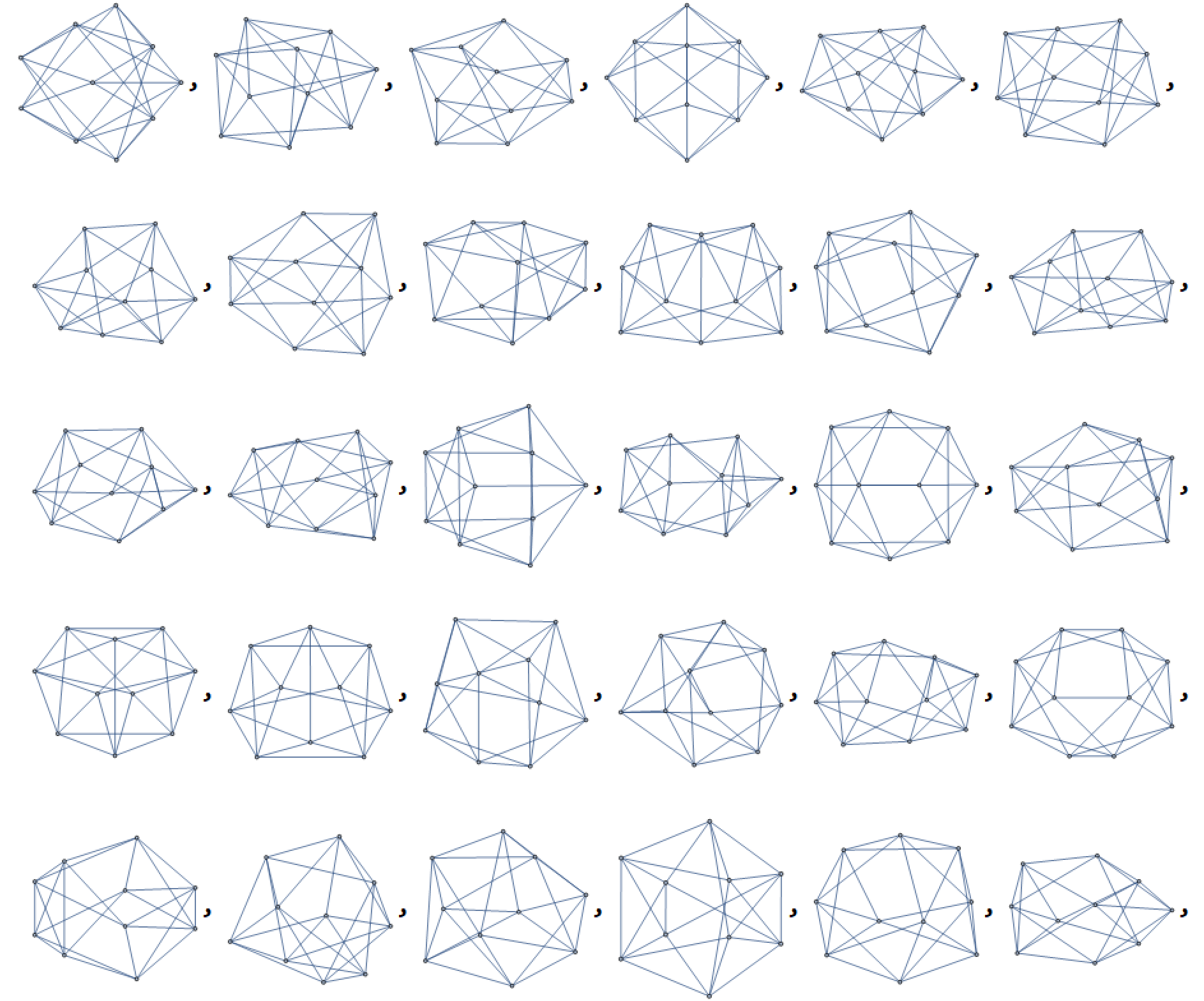}}
  \subfigure{
    \includegraphics[width=12cm]{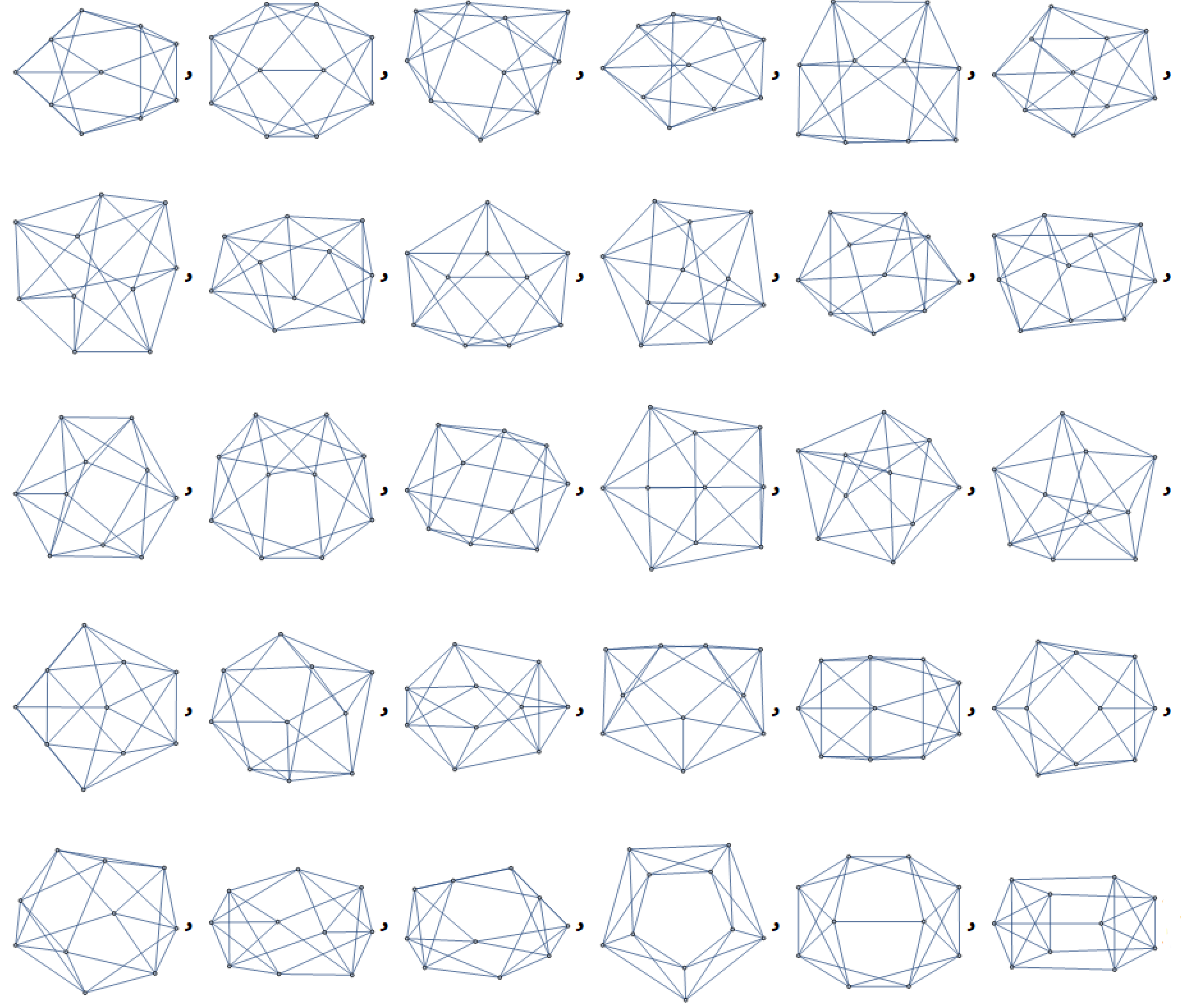}}
  \caption{All 5-regular graphs with order 10}\label{10-5}
\end{figure}

\end{appendices}

\end{document}